\documentclass[aps,preprint,superscriptaddress,showpacs,amsmath,amssymb]{revtex4}

\usepackage{graphicx}
\usepackage{dcolumn}
\usepackage{bm}
\usepackage{amsthm}
\newtheorem{theorem}{Theorem}
\newtheorem{fact}[theorem]{Fact}
\newtheorem{theorem2}{Theorem}
\newtheorem{condition}[theorem2]{Condition}
\newtheorem{theorem3}{Theorem}
\newtheorem{conjecture}[theorem3]{Conjecture}

\begin{document}

\title{Pseudorandom number generator based on the Bernoulli map on cubic algebraic integers}

\author{Asaki Saito}
\email{saito@fun.ac.jp}
\affiliation{Future University Hakodate, 116-2 Kamedanakano-cho, Hakodate, Hokkaido 041-8655, Japan}

\author{Akihiro Yamaguchi}
\affiliation{Fukuoka Institute of Technology, 3-30-1 Wajiro-higashi, Higashi-ku, Fukuoka 811-0295, Japan}

\date{\today}

\begin{abstract}
We develop a method for generating pseudorandom binary sequences using
the Bernoulli map on cubic algebraic integers.
The distinguishing characteristic of our generator is that it generates chaotic true
orbits of the Bernoulli map by exact computation.
In particular, we clarify a way to properly prepare a set of initial
points (i.e., seeds), which is needed when generating multiple
pseudorandom sequences.
With this seed selection method, we can distribute the initial points
almost uniformly in the unit interval and can also guarantee that the
orbits starting from them do not merge.
We also report results of a large variety of tests indicating that the generated pseudorandom sequences have good
statistical properties as well as
an advantage over what is probably the most popular generator, the
Mersenne Twister MT19937.
\end{abstract}


\maketitle

\section{Introduction}\label{sec:Introduction}

A random sequence is a sequence of numbers that are a typical sample
of independently identically distributed random variables, and it
cannot be generated by a deterministic algorithm (cf., e.g.,
Refs.~\cite{Li,Sugita}).
A pseudorandom sequence, i.e., a computer-generated sequence
that appears similar to a random one, is therefore not random at all, but has a wide range of applications, such as Monte
Carlo methods, probabilistic algorithms, and cryptography \cite{Knuth}.
In order to generate such pseudorandom sequences,
various pseudorandom number generators have been proposed,
including linear congruential generators \cite{Lehmer}, linear
feedback shift registers \cite{Golomb}, and the Mersenne Twister
\cite{Matsumoto}.
Of all the generators, MT19937 \cite{Matsumoto}, a version of the
Mersenne Twister, is probably the most popular one at this time.
MT19937 can produce, at very high speed, a pseudorandom sequence
having an ``astronomically'' long period of length $2^{19937}-1$ and having a
high-dimensional (623-dimensional) equidistribution property, which
makes the generator very useful, especially for Monte Carlo
simulations.
Even if a generator has these remarkable properties, however, there is no
guarantee that independence, which is the greatest characteristic of
random sequences, is preferable
(cf. Sec.~\ref{subsec:ComparisonWithMT19937}).
In this paper, we deal with the issue of how we generate pseudorandom
sequences having the best possible statistical properties even if
such generation increases the computational cost to some extent.

Among random sequences, the most fundamental ones are (uniform)
random {\it binary} sequences.
One of the mathematically simplest and soundest ways to generate
(pseudo-) random binary sequences is to use the Bernoulli map.
Also known as the doubling map, the dyadic map,
or the $2 x$ modulo $1$ map, the Bernoulli map
is a map on the half-open unit interval
$\left[0,\,1\right)$ given by
\[
M_B(x)= \left\{
\begin{array}{ll}
2x &   \textrm{~~  if    ~} x \in \left[0,\,1/2\right) \\
2x-1 & \textrm{~~  if    ~} x \in \left[1/2,\,1\right).
\end{array}\right.
\]
Note that the repeated tossing of a fair coin is modeled by the one-sided
Bernoulli shift on $\left\{0,\,1\right\}^{{\mathbb N}}$ with $0$ and
$1$ having equal weight $1/2$, and this Bernoulli shift is
measure-theoretically isomorphic to $M_B$
(cf., e.g., Ref.~\cite{Billingsley}).
Thus, $M_B$ can produce binary sequences
equivalent to those obtained by tossing a fair coin.
However, it is well known that one cannot simulate
$M_B$ with conventional simulation methods
such as those using double-precision binary floating-point numbers or
arbitrary-precision rational numbers (see, e.g.,
Refs.~\cite{AtleeJackson,SaitoPTPS}).
This is because, for $M_B$, finite binary decimals on
$\left[0,\,1\right)$ are eventually fixed points (i.e., points that
reach the fixed point at $x=0$ after finitely many iterations) and
because rational numbers on $\left[0,\,1\right)$ are eventually
periodic points (i.e., points that reach a periodic point after
finitely many iterations).
For this reason, a computational method that realizes pseudorandom
number generation using $M_B$ has not been proposed
\cite{FootnoteNoPRNGwithTentBaker} (except for our previous study
\cite{SaitoChaos2}), although pseudorandom number generators based on
chaotic dynamics have been very widely studied for many decades
\cite{Ulam,LiYorke,Oishi}.

On the other hand, orbit computations using algebraic numbers other
than rational ones have been performed in the fields of number theory
and arithmetic dynamics (e.g.,
Refs.~\cite{Lang,Vivaldi,Lowenstein2,Akiyama,FISTY}).
Also, by using our methods to achieve exact simulations of piecewise
linear and linear fractional maps \cite{SaitoPhysD,SaitoChaos}, one
can generate errorless true orbits displaying the same statistical properties as
typical orbits of $M_B$ (as well as those of the tent map and the
baker's transformation; cf. \cite{FootnoteNoPRNGwithTentBaker}).
In particular, by using true orbits on quadratic algebraic integers,
we succeeded in realizing a pseudorandom number generator using $M_B$
\cite{SaitoChaos2}.
To our knowledge, the generator of Ref.~\cite{SaitoChaos2} is the only
one that has a direct connection to the repeated tossing of a
fair coin, but we can expect that we can establish such generators
having good statistical properties also by using algebraic integers of
degree three or more.
In order to realize such generators, however, we particularly need to
resolve the issue described below.
When proposing a pseudorandom number generator, it is desirable to
simultaneously disclose how one can properly perform seed selection,
especially in the case where one needs multiple seeds to generate more
than one pseudorandom sequence.
In particular, such a method for selecting initial points (i.e.,
seeds) is indispensable for a generator based on true orbits:
In true orbit computations, the longer a true orbit, the higher the
computational cost of generating it.
Therefore, the computational cost can be markedly lowered by
generating a number of relatively short true orbits.
We could establish such a seed selection method in the case of
quadratic algebraic integers, but algebraic numbers of different
degrees are quite distinct from each other, and it is unclear even
whether such a seed selection method exists in the case of algebraic
integers of degree three or more.

In this paper, we realize a pseudorandom number generator using
chaotic true orbits of the Bernoulli map on cubic algebraic integers.
We also devise, for the cubic case, a seed selection method for
generating multiple pseudorandom binary sequences.
Moreover, we demonstrate the ability of our generator by performing
two kinds of computer experiments: extensive statistical
testing and a comparison with MT19937.

\section{Proposed Pseudorandom Number Generator}

In this study, we use cubic algebraic integers to simulate the
Bernoulli map $M_B$.
A cubic algebraic integer is a complex number that is a root of a
monic irreducible cubic polynomial $x^3+bx^2+cx+d$ with $b,c,d \in
{\mathbb Z}$
(see, e.g., Ref.~\cite{Hecke} for a detailed explanation of algebraic
integers).
$M_B$ maps any cubic algebraic integer in the open unit interval
$(0,1)$ to a cubic algebraic integer in $(0,1)$.

Let us introduce two sets, $\bar{S}$ and ${S}$, and a map $\pi$ from
$\bar{S}$ to ${S}$.
Let $\bar{S}$ be the set of all $(b, c, d) \in {\mathbb Z}^3$
satisfying the following three conditions:
\begin{description}
\item[(i)] $b^2-3c \le 0$
\item[(ii)] $d<0$
\item[(iii)] $1+b+c+d>0$
\end{description}
Figure~\ref{fg:SetS} shows part of $\bar{S}$.
If we consider a function $f:{\mathbb R} \rightarrow {\mathbb R}$,
given by $f(x)=x^3+bx^2+cx+d$ with $(b, c, d) \in \bar{S}$, we see
from (i) that $f$ is strictly monotonically increasing.
Thus, $f$ has a unique real root, denoted by $\alpha$.
We also see from (ii) and (iii) that $f(0)<0$ and $f(1)>0$, which
implies $\alpha \in (0,1)$.
Since $\alpha \notin {\mathbb Z}$, we see that $\alpha$ is a cubic
algebraic integer.
Also, let ${S}$ be the set of all cubic algebraic integers in $(0,1)$
that are roots of $x^3+bx^2+cx+d$ with $(b, c, d) \in \bar{S}$.
We can define a map $\pi$ from $\bar{S}$ to ${S}$ by assigning each
$(b, c, d) \in \bar{S}$ the unique real root $\alpha \in {S}$ of
$x^3+bx^2+cx+d$.
It is easy to see that $\pi : \bar{S} \rightarrow {S}$ is a bijection.
In the following, we represent $\alpha \in {S}$ with $(b, c, d) =
\pi^{-1} (\alpha) \in \bar{S}$.
\begin{figure}
\begin{center}
\includegraphics[width=0.5\textwidth,clip]{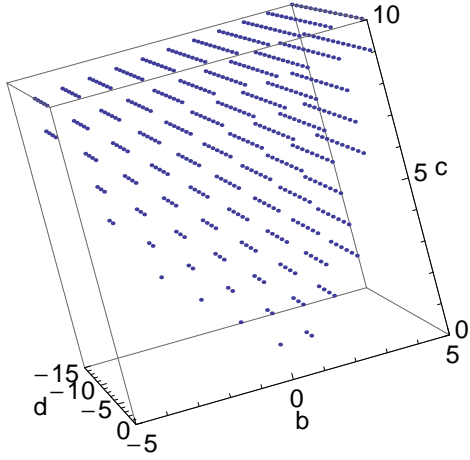}
\end{center}
\caption{\label{fg:SetS}Part of $\bar{S}$. The dots represent
elements of $\bar{S}$.}
\end{figure}

$M_B$ maps $\alpha \in {S}$ to $\alpha'=2\alpha \mbox{ mod } 1$.
As already mentioned, $\alpha'$ is a cubic algebraic integer in
$(0,1)$.
Moreover, we can see $\alpha' \in {S}$ as follows.
Let $(b, c, d)$ be the representation of $\alpha$, and let
$x^3+b'x^2+c'x+d'$ be the minimal polynomial of $\alpha'$.
Then, the coefficients $b'$, $c'$, $d'$ are given as follows:\\
\begin{subequations}
If $\alpha \in (0,1/2)$,
\label{eq:Transformation}
\begin{eqnarray}\label{eq:LeftBranch}
\left(
\begin{array}{c}
b'\\
c'\\
d'
\end{array}
\right)
&=&
\left(
\begin{array}{ccc}
2	&	0 &	0\\
0	&	4 &	0\\
0	&	0 &	8
\end{array}
\right)
\left(
\begin{array}{c}
b\\
c\\
d
\end{array}
\right).
\end{eqnarray}
If $\alpha \in (1/2,1)$,
\begin{eqnarray}\label{eq:RightBranch}
\left(
\begin{array}{c}
b'\\
c'\\
d'
\end{array}
\right)
&=&
\left(
\begin{array}{ccc}
2	&	0 &	0\\
4	&	4 &	0\\
2	&	4 &	8
\end{array}
\right)
\left(
\begin{array}{c}
b\\
c\\
d
\end{array}
\right)
+
\left(
\begin{array}{c}
3\\
3\\
1
\end{array}
\right).
\end{eqnarray}
\end{subequations}
Equation~\eqref{eq:LeftBranch} (resp. Eq.~\eqref{eq:RightBranch}) is
obtained by substituting $\alpha=\alpha'/2$ (resp.
$\alpha=(\alpha'+1)/2$) into $x^3+bx^2+cx+d=0$.
It is easy to confirm that $(b',c',d')$ satisfies the conditions (i),
(ii), and (iii), which implies $\alpha' \in {S}$.

Equation~\eqref{eq:Transformation} gives the explicit form of
$\pi^{-1} \circ M_B \circ \pi$, i.e., the transformation on $\bar{S}$
corresponding to $M_B$.
We denote this transformation by $\bar{M}_B$.
Note that $\bar{M}_B$ gives the representation $(b',c',d')$ of
$\alpha'$ from the representation $(b, c, d)$ of $\alpha$. This
transformation is exactly computable by using only integer arithmetic.
Incidentally, we see easily that $(b, c, d)$ with $b^2-3c < 0$ is
mapped to $(b',c',d')$ with $b'^2-3c' < 0$, and that $(b, c, d)$ with
$b^2-3c = 0$ is mapped to $(b',c',d')$ with $b'^2-3c' = 0$.

One has to exactly determine whether a given $\alpha \in {S}$ is in
$(0,1/2)$ or $(1/2,1)$ in order to generate true orbits of $M_B$ on
${S}$, thereby obtaining pseudorandom binary sequences.
Let $\alpha$ be represented by $(b, c, d) \in \bar{S}$, and let
$f(x)=x^3+bx^2+cx+d$.
This determination can be made easily by evaluating the sign
of $f(1/2)$.
In fact, if $f(1/2)>0$, then $\alpha \in (0,1/2)$;
otherwise, i.e., if $f(1/2)<0$, then $\alpha \in (1/2,1)$.
To evaluate the sign of $f(1/2)$, it is sufficient to evaluate that of
$1+2b+4c+8d$.
Therefore, one can exactly determine whether $\alpha \in (0,1/2)$ or
not by using only integer arithmetic.
Consequently, one can generate a true orbit $\left\{ (b_n, c_n, d_n)
\right\}_{n=0,1,2,\cdots}$ of $\bar{M}_B$ starting from an initial
point $(b_0, c_0, d_0) \in \bar{S}$, where $(b_n, c_n,
d_n)=\bar{M}_B^n(b_0, c_0, d_0)$.
In order to obtain a pseudorandom binary sequence $\left\{ \epsilon_n
\right\}_{n=0,1,2,\cdots}$, all one has to do is let $\epsilon_n=0$ if
$1+2b_n+4c_n+8d_n>0$ and $\epsilon_n=1$ if $1+2b_n+4c_n+8d_n<0$, in
the course of generating a true orbit.

\section{Seed Selection Method}\label{sec:SeedSelectionMethod}

In this section, we consider how to select initial points (i.e.,
seeds).
Because $\alpha$ represented by $(b, c, d) \in \bar{S}$ is irrational,
its binary expansion is guaranteed to be nonperiodic.
Thus, one can choose any $(b, c, d) \in \bar{S}$ as an initial point
in the sense that at least one obtains a nonperiodic binary sequence.
It is worth noting that the binary sequence obtained from $(b, c, d)
\in \bar{S}$ is not only guaranteed to be nonperiodic.
In fact, it is widely believed that every irrational algebraic number
is a normal number (Borel's conjecture \cite{Borel}).
Recall that $\alpha \in {\mathbb R}$ is said to be normal if, for any
integer $b \ge 2$, every word of length $l \ge 1$ on the alphabet
$\{0,1,\ldots,b-1\}$ occurs in the base-$b$ expansion of $\alpha$ with
asymptotic frequency $b^{-l}$.
Also, it is proved that the base-$b$ expansion of any irrational
algebraic number
cannot have a regularity so simple that it can be
generated by a finite automaton \cite{Adamczewski}.
Moreover, our previous studies strongly suggest that the base-$b$
expansion of any irrational algebraic number has the same statistical
properties as those of almost all real numbers
\cite{SaitoPhysD,SaitoChaos,SaitoChaos2}.

For the generation of more than one pseudorandom sequence, it is
necessary to prepare an initial point set $\bar{I} \subset \bar{S}$.
One can consider a variety of conditions that $\bar{I}$ should
satisfy; here, we impose the following two conditions (Conditions
\ref{condition1} and \ref{condition2}) on it.

\begin{condition}\label{condition1}
The elements of ${I} \subset {S}$ corresponding to $\bar{I}$ are
uniformly distributed in the unit interval.
\end{condition}

This condition is for unbiased sampling of initial points
and is a natural one also from the viewpoint of applications, such as
the Monte Carlo method.
However, it is a nontrivial question as to how we can construct
$\bar{I}$ satisfying Condition \ref{condition1}, because $\alpha$
depends on $(b, c, d)$ in a very complicated way.
In fact, $\alpha$ takes the following complex form:\\
\begin{subequations}
If $b^2-3c < 0$,
\begin{eqnarray*}
\alpha &=& \frac{\sqrt[3]{-2 b^3+9 b c-27 d+3 \sqrt{3} \sqrt{-b^2 c^2+4 c^3+4 b^3 d-18 b c d+27 d^2}}}
{3 \sqrt[3]{2}}\\
& & -\frac{\sqrt[3]{2} \left(-b^2+3 c\right)}{3 
\sqrt[3]{-2 b^3+9 b c-27 d+3 \sqrt{3} \sqrt{-b^2 c^2+4 c^3+4 b^3 d-18 b c d+27 d^2}}
}
-\frac{b}{3}.
\end{eqnarray*}
If $b^2-3c = 0$,
\begin{eqnarray*}
\alpha &=& \frac{\sqrt[3]{-2 b^3+9 b c-27 d}}
{3}-\frac{b}{3}.
\end{eqnarray*}
\end{subequations}

\begin{condition}\label{condition2}
The orbits starting from the elements of $\bar{I}$ do not merge.
\end{condition}

Even if one selects two different points as the elements of $\bar{S}$,
the latter parts of the resulting binary sequences may coincide with
each other. In fact, this happens if the two points are on the
same orbit or, more generally, if the orbits starting from them merge.
When generating multiple pseudorandom sequences, it is desirable that
the binary sequences derived from $\bar{I}$ are as different from each
other as possible, and it is obviously desirable that $\bar{I}$
satisfies Condition \ref{condition2}.
However, in order to realize such an $\bar{I}$, we
need to make it clear how we can select the elements of $\bar{I}$
while avoiding such orbital overlaps.

In what follows, we show that we can construct an $\bar{I}$ satisfying
Conditions \ref{condition1} and \ref{condition2}.

Concerning Condition \ref{condition1}, the following fact holds.

\begin{fact}\label{fact:EquidistantlyDistributed}
Let $c$ be a sufficiently large positive integer, and let
\begin{equation}\label{eq:InitialPointSet}
\bar{I}_{b,c}=\left\{(b, c, d) \in \bar{S} ~|~ d \in
\left\{-1,-2,\cdots,-(b+c)\right\} \right\}.
\end{equation}
Then, the elements of ${I}_{b,c} \subset {S}$ corresponding to
$\bar{I}_{b,c}$ are distributed almost uniformly (equidistantly) in
the unit interval.
\end{fact}

\begin{proof}
Since $|b| \le \sqrt{3 c}$, $|b| \ll c$ holds for sufficiently large $c$.
Let $(b, c, d) \in \bar{I}_{b,c}$, $\alpha_{d}= \pi(b, c, d)$, and
$f_{d}(x)=x^3+bx^2+cx+d$.
We see that $f_{-1}(0)=-1$, $\lim_{c \to \infty} f_{-1}(2/c)=1$,
$\lim_{c \to \infty} f_{-(b+c)}(1-2/c)=-1$, and $f_{-(b+c)}(1)=1$.
Thus, we have $\lim_{c \to \infty} \alpha_{-1}=0$ and $\lim_{c \to
  \infty} \alpha_{-(b+c)}=1$.
We also see easily that $\alpha_{d} < \alpha_{d-1}$ and
$f_{d}(\alpha_{d-1})=1$ hold for $d \in
\left\{-1,-2,\cdots,-(b+c)+1\right\}$.
Let $\Delta_{d} = \alpha_{d-1} - \alpha_{d}$ ($d \in
\left\{-1,-2,\cdots,-(b+c)+1\right\}$).
By the mean value theorem, there exists $\beta \in \left(\alpha_{d},
\alpha_{d-1} \right)$ such that $f_{d}'(\beta)=\Delta_{d}^{-1}$.
It is easy to see that $-2 |b| + c < f_{d}'(x) < 3 +2 |b| + c$ holds
for $x \in \left(0, 1 \right)$.
Thus, for sufficiently large $c$, we have $(3 +2 |b| + c)^{-1} <
\Delta_{d} < (-2 |b| + c)^{-1}$, which implies
\[
\left(1+ \frac{2 |b| +3}{c} \right)^{-1} < \frac{\Delta_{d}}{c^{-1}} <
\left(1-\frac{2 |b|}{c}\right)^{-1}.
\]
We note that $-x+2 \le x^{-1}$ holds for $x \ge 1$ and that $x^{-1}
\le -2 x +3$ holds for $1/2 \le x \le 1$.
Thus, for sufficiently large $c$, we have
\[
1- \frac{2 |b| +3}{c} < \frac{\Delta_{d}}{c^{-1}} <
1+ \frac{4 |b|}{c},
\]
which implies
\[
\lim_{c \to \infty} \max_{d \in
\left\{-1,\cdots,-(b+c)+1\right\}} \left| \frac{\Delta_{d}}{c^{-1}} -1 \right|=0.
\]
Therefore, if we take a sufficiently large $c$, the elements of
${I}_{b,c}$ are distributed across the unit interval almost
equidistantly, with distances approximately equal to $c^{-1}$.
\end{proof}

An important characteristic of $\bar{M}_B$ on $\bar{S}$ (or
equivalently, $M_B$ on ${S}$) when considering $\bar{I}$ satisfying
Condition \ref{condition2} is its injectivity.
The inverse image of $(b', c', d') \in \bar{S}$ under $\bar{M}_B$ is
uniquely determined if it exists:
If $b'$ (or $c'$ or $d'$) is even, $(b', c', d')$ is derived from
Eq.~\eqref{eq:LeftBranch}.
If odd, it is derived from Eq.~\eqref{eq:RightBranch}.
Let us call an element of $\bar{S}$ a {\it source point} if it does
not have an inverse image in $\bar{S}$.
It is clear that two different source points do not exist on the same
orbit.
Also, the injectivity prevents the merging of orbits starting from different source points.
Concerning the source points, the following fact holds.

\begin{fact}\label{fact:SourcePoints}
There is no inverse image for $(b, c, d) \in \bar{S}$ if and only if
one of the following conditions holds:
\begin{description}
\item[(i)] $b$, $c$, $d$ are neither all even nor all odd.
\item[(ii)] $b$, $c$, $d$ are all even, but $c \not\equiv 0 \pmod{4}$
or $d \not\equiv 0 \pmod{8}$.
\item[(iii)] $b$, $c$, $d$ are all odd, but $-2b+c \not\equiv 1 \pmod{4}$
or $b-c+d \not\equiv 1 \pmod{8}$.
\end{description}
\end{fact}

\begin{proof}
If $(b, c, d) \in \bar{S}$ has an inverse image, then by
Eq.~\eqref{eq:Transformation} $b$, $c$, $d$ are either all even or all
odd.
We can easily verify that a necessary and sufficient condition for
$(b, c, d) \in \bar{S}$ with $b$, $c$, $d$ all even to have an inverse
image is that both $c \equiv 0 \pmod{4}$ and $d \equiv 0 \pmod{8}$
hold.
Similarly, we can verify that a necessary and sufficient condition for
$(b, c, d) \in \bar{S}$ with $b$, $c$, $d$ all odd to have an inverse
image is that both $-2b+c \equiv 1 \pmod{4}$ and $b-c+d \equiv 1
\pmod{8}$ hold.
Therefore, $(b, c, d) \in \bar{S}$ has no inverse image if and only if
one of the conditions (i)--(iii) holds.
\end{proof}

The orbits starting from the elements of $\bar{I}$ do not merge if one
lets $\bar{I}$ consist of only source points.

Consequently, on the basis of Facts
\ref{fact:EquidistantlyDistributed} and \ref{fact:SourcePoints}, we
can construct $\bar{I}$ satisfying Conditions \ref{condition1} and
\ref{condition2}:
The simplest way is to choose $b$ to be an even integer and $c$ to be
a large positive odd integer, or $b$ to be an odd integer and $c$ to
be a large positive even integer, and to let $\bar{I}$ be the
$\bar{I}_{b,c}$ given by Eq.~\eqref{eq:InitialPointSet}.
Note, however, that consisting of only source points is not a
necessary condition for $\bar{I}$ to be free from orbital mergers.
For example, $\bar{I}_{b,c}$ with $b=0$ and $c=8$ contains a point
that is not a source point, but mergers do not occur with $\bar{I}_{0,8}$ (cf. next
paragraph).

Condition \ref{condition2} is equivalent to the condition that latter
parts of the binary sequences derived from $\bar{I}$ do not coincide,
which in turn is equivalent to the condition that, even if each of the
binary sequences is transformed by any multi-bit shift operation that
is expressible as a map $x \mapsto 2^n x \mbox{ mod } 1$ ($n \in
{\mathbb Z}_{\geq 0}$), none of the resulting sequences are identical.
With computer assistance, one can reveal that many, but not all, of
$\bar{I}_{b,c}$ have a much more desirable property than Condition
\ref{condition2}.
Namely, for many of $\bar{I}_{b,c}$, ${\mathbb Q}(\alpha) \neq
{\mathbb Q}(\beta)$ holds for all $\alpha, \beta \in {I}_{b,c}$ with
$\alpha \neq \beta$ (i.e., each element of ${I}_{b,c}$ belongs to a
different cubic field).
In particular, we experimentally confirmed that all of ${I}_{b,c}$
with $b=0$ and $c$ in $1 \leq c \leq 5 \times 10^4$ have this
desirable property, which leads us to the following conjecture:

\begin{conjecture}\label{conjecture1}
Let $c \in {\mathbb Z}_{> 0}$.
Then, ${\mathbb Q}(\alpha) \neq {\mathbb Q}(\beta)$ holds for all
$\alpha, \beta \in {I}_{0,c}$ with $\alpha \neq \beta$.
\end{conjecture}

If ${I}_{b,c}$ has such a property, the binary sequences derived from
$\bar{I}_{b,c}$ are significantly different from each other in the following
sense:
Even if each of the binary sequences is transformed by any operation
expressible as a rational map with rational coefficients (except those
mapping elements of ${I}_{b,c}$ to rational numbers), the resulting
sequences include no identical sequences.
Such operations include not only multi-bit shifts, but a wide
variety of operations, e.g., all-bit inversion, which is expressible as
the map $x \mapsto 1-x$.

\section{Experimental Results}

\subsection{Statistical testing}

We evaluated our generator using three statistical
test suites: DIEHARD \cite{Marsaglia}, NIST statistical test suite
\cite{NIST}, and TestU01 \cite{L'Ecuyer}.
We summarize their results in Table~\ref{tab:StatisticalTesting}.

We performed DIEHARD and NIST tests on the binary sequences of length
$10^6$ derived from $\bar{I}_{0,1001}$.
For TestU01, we prepared test data as follows:
We generated the binary sequences of length $1000032$ using
$\bar{I}_{0,12000001}$.
We then removed the first 32 bits of each sequence and concatenated
the resulting sequences in descending order of $d$ value.
We removed the first 32-bit blocks in order to avoid introducing
correlations among them, because each of these blocks stores
information regarding the position of the initial point.

Here we briefly explain the three statistical test suites
and report their results.

DIEHARD \cite{Marsaglia} contains 234 statistical tests
classified into 18 categories.
The results for 6 of the 18 categories are further tested by
checking the uniformity of the resulting $P$-values.
(That is, DIEHARD consists of 234 first-level tests and 6 second-level
ones.)
Using DIEHARD version ``DOS, Jan 7, 1997'', we performed all 240 tests
with a significance level of 0.01.
As a result, 238 of the 240 tests were passed.

NIST statistical test suite \cite{NIST} contains 188
first-level tests.
In NIST testing, each of 188 first-level tests is performed $10^3$ times,
and the results of each first-level test are
further tested in two ways:
(i) The proportion of passing sequences
is tested using a significance level of 0.001540
(cf. Ref.~\cite{SaitoChaos2}).
(ii) The uniformity of $P$-values
is tested using a significance level of 0.0001.
For this procedure, we used version 2.1.2 of the NIST statistical
test suite.
As a result, 187 of the 188 second-level tests
based on the proportion of passing sequences were passed.
As for the second-level tests based on the uniformity of $P$-values,
all 188 tests were passed.

TestU01 \cite{L'Ecuyer} offers several predefined sets of tests,
including SmallCrush, Crush, and BigCrush, which consist of 15, 144,
and 160 tests, respectively.
In TestU01, 
the result of each test is interpreted as {\it clear failure}
if the $P$-value for the test is less than $10^{-10}$ or greater than
$1 - 10^{-10}$.
The result is interpreted as {\it suspicious} if the $P$-value is in
$\left[10^{-10}, 10^{-4}\right)$ or $\left(1 - 10^{-4}, 1 -
10^{-10}\right]$.
In all other cases, the test is considered as {\it passed}.
Using version 1.2.3 of TestU01, we applied SmallCrush, Crush, and
BigCrush to the test data described above.
As a result, all tests of SmallCrush, Crush, and BigCrush were passed.

Consequently, all tests were passed for NIST's second-level testing
based on the uniformity of $P$-values and TestU01's SmallCrush, Crush,
and BigCrush, while a few tests were failed for DIEHARD and NIST's
second-level testing based on the proportion of passing sequences.
Note that the numbers of failed tests (i.e., two for DIEHARD
and one for NIST's second-level testing based on the proportion of
passing sequences) are within relevant ranges because they are close
to the expected numbers of failed tests (i.e., 2.40 for DIEHARD
and 0.29 for NIST's second-level testing based on the
proportion of passing sequences).
From these results, we can confirm that our generator has good
statistical properties.
\begin{table}
\caption{\label{tab:StatisticalTesting}Results of statistical testing.}
\begin{ruledtabular}
\begin{tabular}{llcccc}
\multicolumn{2}{c}{Statistical testing}
 &
\multicolumn{4}{c}{Number~of:}\\ \cline{3-6}
 &
 &
Tests &
Passed tests &
Suspicious tests &
Failed tests \\
\colrule
DIEHARD & First-level tests & 234 & 232 & --- & 2 \\
       & Second-level tests & 6 & 6 & --- & 0 \\
NIST STS & Second-level tests (proportion) & 188 & 187 & --- & 1 \\
         & Second-level tests (uniformity) & 188 & 188 & --- & 0 \\
TestU01 & SmallCrush & 15 & 15 & 0 & 0 \\
        & Crush &    144 & 144 & 0 & 0 \\
        & BigCrush & 160 & 160 & 0 & 0
\end{tabular}
\end{ruledtabular}
\end{table}

\subsection{Comparison with the Mersenne Twister MT19937}\label{subsec:ComparisonWithMT19937}

Here we attempt a comparison between our generator and MT19937.

As described in Sec.~\ref{sec:Introduction}, MT19937 is a highly
practical generator that produces, at very high speed, a pseudorandom
sequence having a period of length $2^{19937}-1$ and a 623-dimensional
equidistribution property.
In spite of these marked characteristics, this generator has been
reported to fail linear complexity tests and birthday spacings tests
with specific lags \cite{Panneton,L'Ecuyer,Harase}.
This is due to the fact that the generator is based on a linear
recurrence over the two-element field ${\mathbb
F}_2=\left\{0,\,1\right\}$.

MT19937 generates a sequence of 32-bit unsigned integers.
In the following, we will identify a 32-bit unsigned integer with an
element of ${\mathbb F}_2^{32}$.
Also, we will not distinguish between row and column vectors except in
that a vector postmultiplying a matrix will be regarded as a column
vector.
MT19937 is one of the multiple-recursive matrix methods
\cite{Niederreiter1,Niederreiter2}, and any sequence $\left\{
\mathbf{y}_{n} \right\}_{n=0,1,2,\cdots}$ in ${\mathbb F}_2^{32}$
generated by MT19937 obeys the following recurrence relation
(cf. Ref.~\cite{Harase}):
\begin{equation}\label{eq:MTRecurrenceRelation}
\mathbf{y}_{n}=\mathbf{y}_{n-227} + A \mathbf{y}_{n-623} + B \mathbf{y}_{n-624}, ~~n \ge 624,
\end{equation}
where $\mathbf{y}_{0}, \mathbf{y}_{1}, \cdots \mathbf{y}_{623}$ are
initial vectors, and $A$ and $B$ are $32 \times 32$ matrices with
elements in ${\mathbb F}_2$.
The explicit forms of $A$ and $B$ are given in
Appendix~\ref{Appendix:MTMatrices}.

From Eq.~\eqref{eq:MTRecurrenceRelation}, we can grasp the regularity of
the sequence generated by MT19937.
For example, the most significant 8 bits of $\mathbf{y}_{n}$ and those
of $\mathbf{y}_{n-227}$ coincide if an integer $n$ with $n \ge 624$
satisfies the following two conditions:
\begin{description}
\item[(a)] The inner product of the $i$th row vector of $A$ and
$\mathbf{y}_{n-623}$ equals zero for every $i$ with $1 \le i \le 8$.
\item[(b)] The inner product of
the second row vector of $B$
and $\mathbf{y}_{n-624}$ equals zero.
\end{description}
Note that condition (b) is equivalent to the condition that $B
\mathbf{y}_{n-624}=\mathbf{0}$ (cf. the form of $B$ in
Appendix~\ref{Appendix:MTMatrices}).
Let $\mathbf{y}_n=(y_{n,1},y_{n,2},\cdots,y_{n,32}) \in {\mathbb
F}_2^{32}$ and $Y_{n}=\sum_{i=1}^{8} y_{n,i} 2^{8-i}$ for $n \ge 0$.
We generated a sequence $\left\{ \mathbf{y}_{n}
\right\}_{n=0,1,2,\cdots,312499}$ of 32-bit unsigned integers using
MT19937 \cite{FootnoteInitialization}, and plotted, in
Fig.~\ref{fg:Independence}, the points $(Y_{n-227},Y_{n})$ for $n$
satisfying conditions (a) and (b).
All the points are on the diagonal line $Y_{n}=Y_{n-227}$, but,
obviously, this cannot happen with a random sequence.
\begin{figure}
\begin{center}
\includegraphics[width=0.5\textwidth,clip]{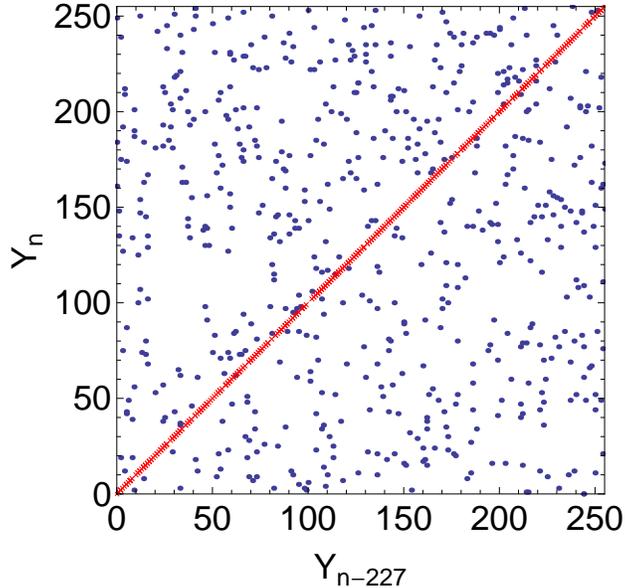}
\end{center}
\caption{\label{fg:Independence}Plot of the points $(Y_{n-227},Y_{n})$ for $n$
satisfying conditions (a) and (b).
Dots represent points obtained from a pseudorandom sequence produced
by our generator.
Crosses represent those by MT19937.}
\end{figure}

On the other hand, our generator produced a binary sequence
of length $10^7$, using $(0, 1, -1) \in \bar{S}$ as an initial point.
Then, by partitioning it into nonoverlapping binary subsequences of
length 32, we transformed it into a sequence $\left\{ \mathbf{y}_{n}
\right\}_{n=0,1,2,\cdots,312499}$ of 32-bit unsigned integers.
Also for this $\left\{ \mathbf{y}_{n}
\right\}_{n=0,1,2,\cdots,312499}$, we plotted, in
Fig.~\ref{fg:Independence}, the points $(Y_{n-227},Y_{n})$ for $n$
satisfying conditions (a) and (b), which was similar to what we did for
$\left\{ \mathbf{y}_{n} \right\}_{n=0,1,2,\cdots,312499}$ obtained by
MT19937.
We can see from Fig.~\ref{fg:Independence} that the points obtained
from our pseudorandom sequence are almost uniformly distributed on the
square.
Although the computational cost of our generator is significantly
higher than that of MT19937, our pseudorandom sequence displays the
same behavior as true (uniform) random sequences.

\section{Conclusion}

In this paper, we have introduced a pseudorandom number generator
using chaotic true orbits of the Bernoulli map on cubic algebraic
integers.
Although this generator has a high computational cost, it exactly
simulates the Bernoulli map that can generate ideal random binary
sequences.
We also have clarified a seed selection method that can select initial
points (i.e., seeds) without bias and can avoid overlaps in latter
parts of the pseudorandom sequences derived from them.
Moreover, we have obtained experimental results supporting the
conjecture that the initial point sets $I_{0,c}$ with $c \in {\mathbb
Z}_{> 0}$ have a more desirable property such that each element of
${I}_{0,c}$ belongs to a different cubic field.
In order to demonstrate the capabilities of our generator, we have
performed two kinds of computer experiments:
Firstly, we have tested our generator using three statistical test
suites---DIEHARD, NIST statistical test suite, and TestU01---and
have shown that it has good statistical properties.
Secondly, we have examined the independence property of pseudorandom
numbers and have clarified an advantage that our generator has over what is probably the most popular generator, the Mersenne Twister
MT19937.

\section*{Acknowledgments}

We thank Shigeki Akiyama, Shunji Ito, Teturo Kamae, Jun-ichi Tamura,
Shin-ichi Yasutomi, and Masamichi Yoshida for their suggestions.
This research was supported by JSPS KAKENHI Grant Number 15K00342.

\appendix

\newpage

\section{Explicit forms of matrices $A$ and $B$ in Eq.~\eqref{eq:MTRecurrenceRelation}}
\label{Appendix:MTMatrices}

In this Appendix, we provide the explicit forms of matrices $A$
and $B$ in the recurrence relation \eqref{eq:MTRecurrenceRelation} for
the Mersenne Twister MT19937.

Matrix $A$:

{\tiny
\[
\left(
\begin{array}{cccccccccccccccccccccccccccccccc}
 0 & 0 & 1 & 1 & 0 & 0 & 0 & 0 & 0 & 1 & 0 & 1 & 0 & 1 & 1 & 0 & 0 & 0 & 0 & 0
   & 1 & 1 & 0 & 0 & 1 & 0 & 0 & 0 & 0 & 1 & 0 & 1 \\[-2.5pt]
 0 & 0 & 1 & 0 & 0 & 0 & 1 & 0 & 0 & 1 & 0 & 0 & 0 & 1 & 0 & 1 & 0 & 0 & 0 & 0
   & 1 & 0 & 0 & 0 & 1 & 0 & 0 & 0 & 0 & 0 & 0 & 1 \\[-2.5pt]
 0 & 1 & 1 & 0 & 0 & 1 & 1 & 0 & 0 & 1 & 0 & 0 & 0 & 1 & 0 & 0 & 0 & 0 & 0 & 0
   & 1 & 0 & 0 & 1 & 1 & 0 & 0 & 0 & 0 & 0 & 0 & 1 \\[-2.5pt]
 0 & 0 & 0 & 0 & 0 & 0 & 0 & 0 & 0 & 0 & 0 & 0 & 0 & 1 & 0 & 0 & 0 & 1 & 0 & 0
   & 1 & 0 & 0 & 0 & 0 & 0 & 0 & 0 & 0 & 0 & 0 & 1 \\[-2.5pt]
 1 & 0 & 1 & 1 & 0 & 0 & 1 & 1 & 0 & 1 & 0 & 0 & 0 & 1 & 0 & 0 & 0 & 0 & 1 & 0
   & 1 & 0 & 0 & 0 & 1 & 1 & 0 & 0 & 0 & 0 & 0 & 1 \\[-2.5pt]
 1 & 0 & 1 & 0 & 0 & 0 & 1 & 1 & 1 & 1 & 0 & 1 & 0 & 1 & 1 & 1 & 0 & 0 & 0 & 0
   & 1 & 0 & 1 & 0 & 1 & 0 & 1 & 0 & 0 & 1 & 0 & 1 \\[-2.5pt]
 0 & 1 & 1 & 0 & 0 & 1 & 1 & 0 & 1 & 0 & 0 & 0 & 1 & 1 & 0 & 0 & 0 & 0 & 0 & 1
   & 1 & 0 & 0 & 0 & 1 & 0 & 1 & 1 & 0 & 0 & 0 & 1 \\[-2.5pt]
 0 & 0 & 1 & 0 & 0 & 0 & 1 & 0 & 0 & 0 & 1 & 0 & 0 & 1 & 0 & 0 & 0 & 0 & 0 & 0
   & 1 & 0 & 0 & 0 & 0 & 0 & 0 & 0 & 1 & 0 & 0 & 0 \\[-2.5pt]
 0 & 0 & 1 & 0 & 0 & 0 & 1 & 1 & 0 & 1 & 0 & 0 & 0 & 1 & 1 & 0 & 0 & 0 & 0 & 0
   & 1 & 0 & 0 & 0 & 1 & 0 & 0 & 0 & 0 & 0 & 0 & 1 \\[-2.5pt]
 0 & 0 & 0 & 0 & 0 & 0 & 0 & 0 & 1 & 0 & 0 & 0 & 1 & 0 & 0 & 0 & 0 & 0 & 0 & 0
   & 0 & 0 & 0 & 0 & 0 & 0 & 0 & 0 & 0 & 0 & 1 & 0 \\[-2.5pt]
 0 & 0 & 0 & 0 & 0 & 0 & 1 & 0 & 0 & 1 & 0 & 0 & 0 & 0 & 0 & 0 & 1 & 0 & 0 & 0
   & 0 & 0 & 0 & 0 & 1 & 0 & 0 & 0 & 0 & 0 & 0 & 0 \\[-2.5pt]
 0 & 0 & 1 & 0 & 0 & 0 & 0 & 0 & 0 & 1 & 1 & 0 & 0 & 0 & 0 & 0 & 0 & 1 & 0 & 0
   & 1 & 0 & 0 & 0 & 0 & 0 & 0 & 0 & 0 & 0 & 0 & 0 \\[-2.5pt]
 0 & 0 & 0 & 0 & 1 & 0 & 0 & 1 & 0 & 0 & 0 & 0 & 0 & 0 & 1 & 1 & 0 & 0 & 1 & 0
   & 0 & 0 & 1 & 0 & 0 & 0 & 0 & 0 & 0 & 1 & 0 & 0 \\[-2.5pt]
 0 & 0 & 1 & 0 & 0 & 0 & 1 & 0 & 0 & 0 & 0 & 0 & 1 & 1 & 0 & 0 & 0 & 0 & 0 & 0
   & 1 & 0 & 0 & 0 & 1 & 0 & 0 & 1 & 0 & 0 & 0 & 1 \\[-2.5pt]
 0 & 0 & 0 & 0 & 0 & 0 & 1 & 0 & 0 & 1 & 0 & 0 & 0 & 0 & 0 & 0 & 0 & 0 & 0 & 0
   & 0 & 0 & 0 & 0 & 1 & 0 & 0 & 0 & 0 & 0 & 0 & 1 \\[-2.5pt]
 0 & 0 & 1 & 0 & 0 & 0 & 1 & 0 & 0 & 1 & 0 & 1 & 0 & 1 & 1 & 0 & 0 & 0 & 0 & 0
   & 1 & 0 & 0 & 0 & 1 & 0 & 0 & 0 & 0 & 1 & 0 & 1 \\[-2.5pt]
 0 & 0 & 1 & 0 & 0 & 0 & 1 & 0 & 0 & 1 & 0 & 0 & 0 & 1 & 0 & 1 & 0 & 0 & 0 & 0
   & 1 & 0 & 0 & 0 & 1 & 0 & 0 & 0 & 0 & 0 & 0 & 1 \\[-2.5pt]
 0 & 0 & 1 & 0 & 0 & 1 & 1 & 0 & 0 & 1 & 0 & 0 & 0 & 1 & 0 & 0 & 1 & 0 & 0 & 0
   & 1 & 0 & 0 & 1 & 1 & 0 & 0 & 0 & 0 & 0 & 0 & 1 \\[-2.5pt]
 0 & 0 & 1 & 1 & 0 & 0 & 1 & 0 & 0 & 1 & 0 & 1 & 0 & 0 & 1 & 0 & 0 & 1 & 0 & 0
   & 1 & 1 & 0 & 0 & 0 & 0 & 0 & 0 & 0 & 1 & 0 & 0 \\[-2.5pt]
 1 & 0 & 0 & 0 & 0 & 0 & 0 & 1 & 0 & 0 & 0 & 0 & 0 & 0 & 0 & 1 & 0 & 0 & 1 & 0
   & 0 & 0 & 0 & 0 & 0 & 1 & 0 & 0 & 0 & 0 & 0 & 0 \\[-2.5pt]
 0 & 0 & 1 & 0 & 0 & 1 & 1 & 0 & 1 & 1 & 0 & 0 & 0 & 1 & 0 & 0 & 0 & 0 & 0 & 1
   & 1 & 0 & 0 & 1 & 1 & 0 & 1 & 0 & 0 & 0 & 0 & 1 \\[-2.5pt]
 0 & 0 & 0 & 0 & 0 & 0 & 1 & 0 & 0 & 0 & 0 & 0 & 0 & 0 & 0 & 0 & 0 & 1 & 0 & 0
   & 1 & 0 & 0 & 0 & 1 & 0 & 0 & 1 & 0 & 0 & 0 & 0 \\[-2.5pt]
 1 & 0 & 1 & 0 & 0 & 0 & 1 & 1 & 0 & 1 & 0 & 0 & 0 & 1 & 0 & 0 & 0 & 0 & 1 & 0
   & 1 & 1 & 0 & 0 & 1 & 1 & 0 & 0 & 0 & 0 & 0 & 1 \\[-2.5pt]
 1 & 0 & 0 & 0 & 1 & 0 & 0 & 1 & 1 & 0 & 0 & 0 & 0 & 0 & 1 & 1 & 0 & 0 & 0 & 0
   & 0 & 0 & 0 & 0 & 0 & 0 & 1 & 0 & 0 & 0 & 0 & 0 \\[-2.5pt]
 0 & 1 & 1 & 0 & 0 & 0 & 1 & 0 & 1 & 0 & 0 & 0 & 0 & 1 & 0 & 0 & 0 & 0 & 0 & 1
   & 1 & 0 & 0 & 1 & 1 & 0 & 1 & 1 & 0 & 0 & 1 & 1 \\[-2.5pt]
 0 & 0 & 0 & 0 & 0 & 0 & 1 & 0 & 0 & 1 & 1 & 0 & 0 & 1 & 0 & 0 & 0 & 0 & 0 & 0
   & 0 & 0 & 0 & 0 & 0 & 0 & 0 & 0 & 1 & 0 & 0 & 0 \\[-2.5pt]
 0 & 0 & 1 & 0 & 0 & 0 & 1 & 0 & 0 & 1 & 0 & 0 & 0 & 1 & 1 & 0 & 0 & 0 & 0 & 0
   & 1 & 0 & 0 & 0 & 1 & 1 & 0 & 0 & 0 & 0 & 0 & 1 \\[-2.5pt]
 0 & 0 & 0 & 0 & 0 & 0 & 0 & 0 & 0 & 0 & 0 & 0 & 1 & 0 & 0 & 0 & 0 & 0 & 0 & 0
   & 0 & 0 & 0 & 0 & 0 & 0 & 1 & 0 & 0 & 0 & 1 & 0 \\[-2.5pt]
 0 & 0 & 1 & 0 & 0 & 0 & 0 & 0 & 0 & 1 & 0 & 0 & 0 & 1 & 0 & 0 & 1 & 0 & 0 & 0
   & 1 & 0 & 0 & 0 & 0 & 0 & 0 & 1 & 0 & 0 & 0 & 1 \\[-2.5pt]
 0 & 0 & 1 & 0 & 0 & 0 & 0 & 0 & 0 & 1 & 0 & 0 & 0 & 0 & 0 & 0 & 0 & 1 & 0 & 0
   & 1 & 0 & 0 & 0 & 0 & 0 & 0 & 0 & 1 & 0 & 0 & 0 \\[-2.5pt]
 0 & 0 & 0 & 0 & 1 & 0 & 0 & 1 & 0 & 0 & 0 & 1 & 0 & 0 & 1 & 1 & 0 & 0 & 1 & 0
   & 0 & 0 & 1 & 0 & 0 & 0 & 0 & 0 & 0 & 0 & 0 & 0 \\[-2.5pt]
 0 & 0 & 0 & 0 & 0 & 0 & 0 & 0 & 0 & 1 & 0 & 0 & 0 & 0 & 0 & 0 & 0 & 0 & 0 & 0
   & 0 & 0 & 0 & 0 & 0 & 0 & 0 & 1 & 0 & 0 & 1 & 0 \\
\end{array}
\right)
\]
}

Matrix $B$:

{\tiny
\[
\left(
\begin{array}{cccccccccccccccccccccccccccccccc}
 0 & 0 & 0 & 0 & 0 & 0 & 0 & 0 & 0 & 0 & 0 & 0 & 0 & 0 & 0 & 0 & 0 & 0 & 0 & 0
   & 0 & 0 & 0 & 0 & 0 & 0 & 0 & 0 & 0 & 0 & 0 & 0 \\[-2.5pt]
 1 & 0 & 0 & 0 & 1 & 0 & 0 & 1 & 0 & 0 & 0 & 1 & 0 & 0 & 1 & 1 & 0 & 0 & 0 & 0
   & 0 & 0 & 1 & 0 & 0 & 0 & 0 & 0 & 0 & 1 & 0 & 0 \\[-2.5pt]
 0 & 0 & 0 & 0 & 0 & 0 & 0 & 0 & 0 & 0 & 0 & 0 & 0 & 0 & 0 & 0 & 0 & 0 & 0 & 0
   & 0 & 0 & 0 & 0 & 0 & 0 & 0 & 0 & 0 & 0 & 0 & 0 \\[-2.5pt]
 0 & 0 & 0 & 0 & 0 & 0 & 0 & 0 & 0 & 0 & 0 & 0 & 0 & 0 & 0 & 0 & 0 & 0 & 0 & 0
   & 0 & 0 & 0 & 0 & 0 & 0 & 0 & 0 & 0 & 0 & 0 & 0 \\[-2.5pt]
 0 & 0 & 0 & 0 & 0 & 0 & 0 & 0 & 0 & 0 & 0 & 0 & 0 & 0 & 0 & 0 & 0 & 0 & 0 & 0
   & 0 & 0 & 0 & 0 & 0 & 0 & 0 & 0 & 0 & 0 & 0 & 0 \\[-2.5pt]
 1 & 0 & 0 & 0 & 1 & 0 & 0 & 1 & 0 & 0 & 0 & 1 & 0 & 0 & 1 & 1 & 0 & 0 & 0 & 0
   & 0 & 0 & 1 & 0 & 0 & 0 & 0 & 0 & 0 & 1 & 0 & 0 \\[-2.5pt]
 0 & 0 & 0 & 0 & 0 & 0 & 0 & 0 & 0 & 0 & 0 & 0 & 0 & 0 & 0 & 0 & 0 & 0 & 0 & 0
   & 0 & 0 & 0 & 0 & 0 & 0 & 0 & 0 & 0 & 0 & 0 & 0 \\[-2.5pt]
 0 & 0 & 0 & 0 & 0 & 0 & 0 & 0 & 0 & 0 & 0 & 0 & 0 & 0 & 0 & 0 & 0 & 0 & 0 & 0
   & 0 & 0 & 0 & 0 & 0 & 0 & 0 & 0 & 0 & 0 & 0 & 0 \\[-2.5pt]
 0 & 0 & 0 & 0 & 0 & 0 & 0 & 0 & 0 & 0 & 0 & 0 & 0 & 0 & 0 & 0 & 0 & 0 & 0 & 0
   & 0 & 0 & 0 & 0 & 0 & 0 & 0 & 0 & 0 & 0 & 0 & 0 \\[-2.5pt]
 0 & 0 & 0 & 0 & 0 & 0 & 0 & 0 & 0 & 0 & 0 & 0 & 0 & 0 & 0 & 0 & 0 & 0 & 0 & 0
   & 0 & 0 & 0 & 0 & 0 & 0 & 0 & 0 & 0 & 0 & 0 & 0 \\[-2.5pt]
 0 & 0 & 0 & 0 & 0 & 0 & 0 & 0 & 0 & 0 & 0 & 0 & 0 & 0 & 0 & 0 & 0 & 0 & 0 & 0
   & 0 & 0 & 0 & 0 & 0 & 0 & 0 & 0 & 0 & 0 & 0 & 0 \\[-2.5pt]
 0 & 0 & 0 & 0 & 0 & 0 & 0 & 0 & 0 & 0 & 0 & 0 & 0 & 0 & 0 & 0 & 0 & 0 & 0 & 0
   & 0 & 0 & 0 & 0 & 0 & 0 & 0 & 0 & 0 & 0 & 0 & 0 \\[-2.5pt]
 1 & 0 & 0 & 0 & 1 & 0 & 0 & 1 & 0 & 0 & 0 & 1 & 0 & 0 & 1 & 1 & 0 & 0 & 0 & 0
   & 0 & 0 & 1 & 0 & 0 & 0 & 0 & 0 & 0 & 1 & 0 & 0 \\[-2.5pt]
 0 & 0 & 0 & 0 & 0 & 0 & 0 & 0 & 0 & 0 & 0 & 0 & 0 & 0 & 0 & 0 & 0 & 0 & 0 & 0
   & 0 & 0 & 0 & 0 & 0 & 0 & 0 & 0 & 0 & 0 & 0 & 0 \\[-2.5pt]
 0 & 0 & 0 & 0 & 0 & 0 & 0 & 0 & 0 & 0 & 0 & 0 & 0 & 0 & 0 & 0 & 0 & 0 & 0 & 0
   & 0 & 0 & 0 & 0 & 0 & 0 & 0 & 0 & 0 & 0 & 0 & 0 \\[-2.5pt]
 0 & 0 & 0 & 0 & 0 & 0 & 0 & 0 & 0 & 0 & 0 & 0 & 0 & 0 & 0 & 0 & 0 & 0 & 0 & 0
   & 0 & 0 & 0 & 0 & 0 & 0 & 0 & 0 & 0 & 0 & 0 & 0 \\[-2.5pt]
 0 & 0 & 0 & 0 & 0 & 0 & 0 & 0 & 0 & 0 & 0 & 0 & 0 & 0 & 0 & 0 & 0 & 0 & 0 & 0
   & 0 & 0 & 0 & 0 & 0 & 0 & 0 & 0 & 0 & 0 & 0 & 0 \\[-2.5pt]
 0 & 0 & 0 & 0 & 0 & 0 & 0 & 0 & 0 & 0 & 0 & 0 & 0 & 0 & 0 & 0 & 0 & 0 & 0 & 0
   & 0 & 0 & 0 & 0 & 0 & 0 & 0 & 0 & 0 & 0 & 0 & 0 \\[-2.5pt]
 0 & 0 & 0 & 0 & 0 & 0 & 0 & 0 & 0 & 0 & 0 & 0 & 0 & 0 & 0 & 0 & 0 & 0 & 0 & 0
   & 0 & 0 & 0 & 0 & 0 & 0 & 0 & 0 & 0 & 0 & 0 & 0 \\[-2.5pt]
 1 & 0 & 0 & 0 & 1 & 0 & 0 & 1 & 0 & 0 & 0 & 1 & 0 & 0 & 1 & 1 & 0 & 0 & 0 & 0
   & 0 & 0 & 1 & 0 & 0 & 0 & 0 & 0 & 0 & 1 & 0 & 0 \\[-2.5pt]
 0 & 0 & 0 & 0 & 0 & 0 & 0 & 0 & 0 & 0 & 0 & 0 & 0 & 0 & 0 & 0 & 0 & 0 & 0 & 0
   & 0 & 0 & 0 & 0 & 0 & 0 & 0 & 0 & 0 & 0 & 0 & 0 \\[-2.5pt]
 0 & 0 & 0 & 0 & 0 & 0 & 0 & 0 & 0 & 0 & 0 & 0 & 0 & 0 & 0 & 0 & 0 & 0 & 0 & 0
   & 0 & 0 & 0 & 0 & 0 & 0 & 0 & 0 & 0 & 0 & 0 & 0 \\[-2.5pt]
 0 & 0 & 0 & 0 & 0 & 0 & 0 & 0 & 0 & 0 & 0 & 0 & 0 & 0 & 0 & 0 & 0 & 0 & 0 & 0
   & 0 & 0 & 0 & 0 & 0 & 0 & 0 & 0 & 0 & 0 & 0 & 0 \\[-2.5pt]
 1 & 0 & 0 & 0 & 1 & 0 & 0 & 1 & 0 & 0 & 0 & 1 & 0 & 0 & 1 & 1 & 0 & 0 & 0 & 0
   & 0 & 0 & 1 & 0 & 0 & 0 & 0 & 0 & 0 & 1 & 0 & 0 \\[-2.5pt]
 0 & 0 & 0 & 0 & 0 & 0 & 0 & 0 & 0 & 0 & 0 & 0 & 0 & 0 & 0 & 0 & 0 & 0 & 0 & 0
   & 0 & 0 & 0 & 0 & 0 & 0 & 0 & 0 & 0 & 0 & 0 & 0 \\[-2.5pt]
 0 & 0 & 0 & 0 & 0 & 0 & 0 & 0 & 0 & 0 & 0 & 0 & 0 & 0 & 0 & 0 & 0 & 0 & 0 & 0
   & 0 & 0 & 0 & 0 & 0 & 0 & 0 & 0 & 0 & 0 & 0 & 0 \\[-2.5pt]
 0 & 0 & 0 & 0 & 0 & 0 & 0 & 0 & 0 & 0 & 0 & 0 & 0 & 0 & 0 & 0 & 0 & 0 & 0 & 0
   & 0 & 0 & 0 & 0 & 0 & 0 & 0 & 0 & 0 & 0 & 0 & 0 \\[-2.5pt]
 0 & 0 & 0 & 0 & 0 & 0 & 0 & 0 & 0 & 0 & 0 & 0 & 0 & 0 & 0 & 0 & 0 & 0 & 0 & 0
   & 0 & 0 & 0 & 0 & 0 & 0 & 0 & 0 & 0 & 0 & 0 & 0 \\[-2.5pt]
 0 & 0 & 0 & 0 & 0 & 0 & 0 & 0 & 0 & 0 & 0 & 0 & 0 & 0 & 0 & 0 & 0 & 0 & 0 & 0
   & 0 & 0 & 0 & 0 & 0 & 0 & 0 & 0 & 0 & 0 & 0 & 0 \\[-2.5pt]
 0 & 0 & 0 & 0 & 0 & 0 & 0 & 0 & 0 & 0 & 0 & 0 & 0 & 0 & 0 & 0 & 0 & 0 & 0 & 0
   & 0 & 0 & 0 & 0 & 0 & 0 & 0 & 0 & 0 & 0 & 0 & 0 \\[-2.5pt]
 1 & 0 & 0 & 0 & 1 & 0 & 0 & 1 & 0 & 0 & 0 & 1 & 0 & 0 & 1 & 1 & 0 & 0 & 0 & 0
   & 0 & 0 & 1 & 0 & 0 & 0 & 0 & 0 & 0 & 1 & 0 & 0 \\[-2.5pt]
 0 & 0 & 0 & 0 & 0 & 0 & 0 & 0 & 0 & 0 & 0 & 0 & 0 & 0 & 0 & 0 & 0 & 0 & 0 & 0
   & 0 & 0 & 0 & 0 & 0 & 0 & 0 & 0 & 0 & 0 & 0 & 0 \\
\end{array}
\right)
\]
}


\begin{thebibliography}{100}
\bibitem{FootnoteNoPRNGwithTentBaker}
For a similar reason, there has also been no proposed pseudorandom number
generator using the tent map on $\left[0,\,1\right]$ given by $M_T(x)
= 1-|2 x - 1|$ or the baker's transformation on $[0,1)^2$ given by
\[
M_b(x,y)= \left\{
\begin{array}{lc}
\left( 2x, \displaystyle\frac{y}{2} \right) & \textrm{  if    ~} x \in \left[0,\,1/2\right) \\
\left( 2x-1, \displaystyle\frac{y+1}{2} \right) & \textrm{  if    ~} x \in \left[1/2,\,1\right)
\end{array}\right.,
\]
although these maps, together with $M_B$, provide (literally) textbook
examples of chaotic maps.
\end{thebibliography}
\end{document}